\documentclass[12pt]{amsart}

\usepackage{amsmath,bm,amssymb,amsthm,textcomp}
\usepackage{enumitem}
\setlist[enumerate]{beginpenalty=10000}
\usepackage{mathtools}
\usepackage{hyperref}
\usepackage[numbers,sort&compress]{natbib}

\DeclareMathOperator{\dimH}{dim_H}

\theoremstyle{plain}
\newtheorem{theorem}{Theorem}[section]

\newtheorem{proposition}[theorem]{Proposition}
\newtheorem{corollary}[theorem]{Corollary}

\theoremstyle{definition}

\theoremstyle{remark}
\newtheorem{remark}[theorem]{Remark}

\numberwithin{equation}{section}

\newcommand{\I}{\mathbb{I}}
\newcommand{\Is}{\mathbb{I}^{\ast}}
\newcommand{\R}{\mathbb{R}}
\newcommand{\N}{\mathbb{N}}
\newcommand{\Q}{\mathbb{Q}}
\newcommand{\E}{\Sigma_E}
\newcommand{\Pp}{\Sigma_P}
\newcommand{\Pup}{\Sigma_P^{\uparrow}}
\newcommand{\cont}{\mathfrak c}
\newcommand{\cT}{\mathcal T}
\newcommand{\cF}{\mathcal F}
\newcommand{\cB}{\mathcal B}
\newcommand{\cD}{\mathcal D}
\newcommand{\sP}{\mathsf P}
\newcommand{\SigmaB}[1]{\boldsymbol{\Sigma}^0_{#1}}
\newcommand{\PiB}[1]{\boldsymbol{\Pi}^0_{#1}}

\allowdisplaybreaks

\begin{document}

\title[Baire-category transfer between Engel and Pierce expansions]{Baire-category transfer between Engel and Pierce expansions}

\author{Min Woong Ahn}

\address{Department of Mathematics Education, Dankook University, 152, Jukjeon-ro, Suji-gu, Yongin-si, Gyeonggi-do 16890, Republic of Korea}

\email{mwahn@dankook.ac.kr}



\begin{abstract}
In this paper, we construct an explicit correspondence between Engel and Pierce expansion digit sequences, obtained by adding $n-2$ to the $n$th Engel expansion digit, and show that it induces a homeomorphism from the set of irrationals in the unit interval onto the complement of a countable dense subset of this set. This yields a transfer principle for Baire category between the two expansions, and we determine exactly which Borel classes are preserved. We then identify this homeomorphism with the composition of two maps considered by Moroz (2027), and study the continuity of the second map, which sends modified Engel expansions to Pierce expansions. We determine its set of discontinuities, show that every discontinuity is a jump, and prove that the map is nowhere monotone. In particular, this confirms, for this map, the nowhere monotonicity and the continuity outside the exceptional set conjectured by Moroz (2027). As applications, we characterize the non-increasing position-dependent weights for which the associated series of Pierce expansion digits diverges on a comeager set, and we relate the convergence exponents of the Engel and Pierce expansion digit sequences.
\end{abstract}
\maketitle

\tableofcontents

\section{Introduction}\label{sec:intro}

Let $\I\coloneqq(0,1)\setminus\Q$ denote the set of irrational numbers in $(0,1)$. It is well known that every $x\in\I$ admits a unique Engel expansion and a unique Pierce expansion of the form
\begin{equation}\label{eq:expansions}
 x=\sum_{n=1}^{\infty}\frac{1}{e_1(x)\cdots e_n(x)}
 =\sum_{n=1}^{\infty}\frac{(-1)^{n-1}}{p_1(x)\cdots p_n(x)},
\end{equation}
where the Engel expansion digit sequence $(e_n(x))_{n\ge1}$ is non-decreasing with $e_1(x)\ge2$ and $e_n(x)\to\infty$ as $n\to\infty$, and the Pierce expansion digit sequence $(p_n(x))_{n\ge1}$ is strictly increasing. We refer the reader to \cite{Gal76} for basic facts about Engel expansions and to \cite{Pie29,Sha86} for those about Pierce expansions.

Both expansions have been extensively studied from the metric viewpoint, and the two theories have developed in parallel. For instance, the law of large numbers for Engel expansions, which states that $(e_n(x))^{1/n}\to e$ as $n\to\infty$ for Lebesgue-almost every $x\in\I$, was proved by L\'evy \cite{Lev47} (see also \cite{ERS58}). About four decades later, Shallit \cite{Sha86} established the analogous result for Pierce expansions. A similar pattern can be observed in the study of the Hausdorff dimension; compare, e.g., \cite{LW01,SW21} for Engel expansions with \cite{AhnJFG,AhnAA,AhnIJNT,LLS26} for Pierce expansions. The same holds for Baire category; compare, e.g., \cite{AhnEngel,FS24,SW21} for Engel expansions with \cite{AhnIJNT} for Pierce expansions.

Recently, Moroz \cite{Mor} explained part of this parallel at the level of the Hausdorff dimension. In the framework of Perron expansions \cite{Mor24,Mor25}, a map $\cT$ from classical Engel expansions to modified Engel expansions and a map $\cF$ from modified Engel expansions to Pierce expansions were studied in \cite{Mor}, where it was shown that $\cF$ preserves the Hausdorff dimension, while $\cT$ does so under additional growth conditions on the digits (see \cite[Theorems~4.1 and~4.5]{Mor}). Earlier, it was shown in \cite[Theorem~5.1]{Mor25} that $\cF$ preserves the Lebesgue measure. On $\I$, the composition $\cF\circ\cT$ is precisely the coordinatewise shift
\begin{equation}\label{eq:digit-shift}
 p_n=e_n+n-2\qquad(n\in\N);
\end{equation}
see Proposition~\ref{prop:moroz-composition}.

In this paper, we study the correspondence \eqref{eq:digit-shift} from a purely topological viewpoint, focusing on Baire category. Let $C\subseteq\I$ be the set of irrational numbers whose Pierce expansion digits are eventually consecutive, i.e., $p_{n+1}=p_n+1$ for all sufficiently large $n$. We show that $C$ is countable and dense in $\I$, and that the map
\[
 \Theta\colon\I\to\I,\qquad
 p_n(\Theta(x)) \coloneqq e_n(x)+n-2,
\]
is a homeomorphism of $\I$ onto $\I\setminus C$. Moreover, $\Theta$ maps each Engel fundamental interval onto the corresponding Pierce fundamental interval with $C$ removed. Consequently, Baire category results on the Engel side pass through $\Theta$ to the Pierce side, and vice versa. We also show that $\Theta$ preserves the Borel classes $\PiB{\alpha}$ for $\alpha\ge2$ and $\SigmaB{\alpha}$ for $\alpha\ge3$, and that these ranges are optimal (Theorem~\ref{thm:transfer}).

We remark that the existence of a homeomorphism between $\I$ and $\I\setminus C$ is not surprising in itself: both are non-empty zero-dimensional Polish spaces in which every compact subset has empty interior, and hence both are homeomorphic to the Baire space $\N^{\N}$ by the Alexandrov--Urysohn theorem (see \cite[Theorem~7.7]{Ke}). However, $\Theta$ acts coordinatewise on the digit sequences, and so explicit conditions on the Engel expansion digits correspond to explicit conditions on the Pierce expansion digits.

The map $\Theta$ also sheds light on the maps considered by Moroz. In \cite[Section~4.1]{Mor}, Moroz conjectured that, for every sequence $\sP$ (denoted by $P$ in \cite{Mor}) determining a pair of positive and alternating Perron expansions, the associated map $\cF_{\sP}$ is nowhere monotone, has jump discontinuities at points of a countable exceptional set $IS^{\sP^-}$, and is continuous elsewhere. Our map $\cF$ is $\cF_{\sP}$ for a particular choice of $\sP$ (see Section~\ref{sec:moroz}), and in this case, we have $IS^{\sP^-}=\Q\cap(0,1]$. Let $\cB$ be the set of right endpoints of the modified Engel fundamental intervals, which is a countable dense subset of $\Q\cap(0,1]$. We show that $\cF$ is left-continuous everywhere and nowhere monotone, that its set of discontinuities is exactly $\cB\setminus\{1\}$, and that every discontinuity is a jump whose right-hand limit is rational (Theorem~\ref{thm:moroz}). In particular, $\cF$ is continuous at every irrational point. This proves the conjectured nowhere monotonicity and continuity outside $IS^{\sP^-}$, and shows that all discontinuities are jumps at points of $IS^{\sP^-}$. Whether every point of $IS^{\sP^-}\setminus\{1\}$ is a discontinuity remains open (see Remark~\ref{rem:B}).

As applications, we consider series of the form $\sum_{n\ge1} w_n(p_n)$, where each $w_n\colon\N\to[0,\infty)$ is a function, so that the term involving the $n$th digit may depend on the position $n$ as well as on the digit itself. We show that such a series diverges on a comeager subset of $\I$ whenever $\sum_{n\ge1} w_n(n+c)=\infty$ for arbitrarily large integers $c\ge0$, and that this condition is also necessary for denseness when each $w_n$ is non-increasing (Theorem~\ref{thm:pierce-series}). In particular, taking $w_n(t)\coloneqq t^{-1}$ (see Corollary~\ref{cor:g}), we recover the fact, established in \cite[Theorem~1.13]{AhnIJNT}, that the series $\sum_{n\ge1} p_n^{-1}$ diverges on a dense $G_\delta$ subset of $\I$. Finally, we show that $\lambda_P\circ\Theta=\min\{\lambda_E,1\}$, where $\lambda_E$ and $\lambda_P$ denote the convergence exponents of the Engel and Pierce expansion digit sequences, respectively (Proposition~\ref{prop:exponent}). This explains why, in the sense of Baire category, the convergence exponent is typically $\infty$ on the Engel side \cite[Theorem~3.5]{SW21} but $1$ on the Pierce side \cite[Theorem~1.8]{AhnIJNT}.

This paper is organized as follows. In Section~\ref{sec:prelim}, we introduce the symbolic spaces for Engel and Pierce expansions, as well as modified Engel expansions and Moroz's maps $\cT$ and $\cF$. In Section~\ref{sec:main}, we state and prove the main results of this paper, namely, the Baire-category transfer (Theorem~\ref{thm:transfer}), the continuity properties of $\cF$ (Theorem~\ref{thm:moroz}), and the criteria for the comeager divergence of digit series (Theorems~\ref{thm:engel-series} and~\ref{thm:pierce-series}). In Section~\ref{sec:consequences}, we present some consequences of the main results for series of Pierce expansion digits and for convergence exponents.

Throughout the paper, we denote by $\N$ the set of positive integers and by $\N_0\coloneqq\N\cup\{0\}$ the set of non-negative integers. The set $\I$ is endowed with the subspace topology induced by the usual topology on $\R$, and, unless otherwise stated, the terms dense, meager, comeager, and nonmeager refer to $\I$. For a subset $F$ of $\R$, we denote by $|F|$ its diameter and by $\dimH F$ its Hausdorff dimension. We denote by $\cont$ the cardinality of $\R$. For a metrizable space $X$ and a countable ordinal $\alpha\ge1$, we denote by $\SigmaB{\alpha}(X)$ and $\PiB{\alpha}(X)$ the Borel classes of $X$ (see \cite[Section~11.B]{Ke}). Following the usual convention, we define $\inf\varnothing\coloneqq\infty$, and empty sums and empty products are $0$ and $1$, respectively. The logarithm $\log$ always means the natural logarithm.

\section{Preliminaries}\label{sec:prelim}

In this section, we introduce the symbolic spaces for Engel and Pierce expansions, as well as modified Engel expansions and Moroz's maps, which will be used in Section~\ref{sec:main}.

\subsection{Symbolic spaces}\label{subsec:symbolic}

Equip $\N$ with the discrete topology and $\N^{\N}$ with the product topology, and put
\begin{align*}
 \E&\coloneqq\{\eta\in\N^{\N}:2\le\eta_1\le\eta_2\le\cdots\text{ and }\eta_n\to\infty\text{ as }n\to\infty\},\\
 \Pp&\coloneqq\{\pi\in\N^{\N}:1\le\pi_1<\pi_2<\cdots\}.
\end{align*}
A finite sequence $a\coloneqq(a_1,\dotsc,a_k)$ of positive integers, where $k\in\N_0$, is called an \emph{Engel admissible word of length $k$} if $2\le a_1\le\dotsb\le a_k$, and a \emph{Pierce admissible word of length $k$} if $a_1<\dotsb<a_k$. For $\star\in\{E,P\}$ and an admissible word $a$ of length $k$ of the corresponding type, the \emph{cylinder set} associated with $a$ is defined by
\[
 [a]_\star\coloneqq\{\omega\in\Sigma_\star:\omega_j=a_j\text{ for all }1\le j\le k\}.
\]
Each cylinder set is non-empty and clopen, and the cylinder sets form a base of $\Sigma_\star$. Let $f_\star\colon\I\to\Sigma_\star$ and $\varphi_\star\colon\Sigma_\star\to\I$ denote the digit coding and evaluation maps, respectively, i.e., $f_E(x)\coloneqq(e_n(x))_{n\ge1}$ and $f_P(x)\coloneqq(p_n(x))_{n\ge1}$, and $\varphi_E$ and $\varphi_P$ are given by the two series in \eqref{eq:expansions}. We call $I_\star(a)\coloneqq\varphi_\star([a]_\star)$ the \emph{fundamental interval} associated with $a$. Note that it consists of irrational numbers only. For a Pierce admissible word $\sigma$ of length $k$, put
\[
 r(\sigma)\coloneqq\sum_{n=1}^{k}\frac{(-1)^{n-1}}{\sigma_1\cdots\sigma_n}\in\Q .
\]

In the notation of \cite{AhnEngel} and \cite{AhnJFG}, $\E$ and $\Pp$ are the sets $\Sigma_E^{\mathrm{irr}}$ and $\Sigma_\infty$, respectively, endowed with the subspace topologies. The following proposition is \cite[Corollary~3.5]{AhnEngel} for Engel expansions, and it follows from \cite[Proposition~2.3 and Lemmas~3.8 and~3.9]{AhnJFG} for Pierce expansions.

\begin{proposition}\label{prop:coding}
The digit coding maps $f_E\colon\I\to\E$ and $f_P\colon\I\to\Pp$ are homeomorphisms with inverses $\varphi_E$ and $\varphi_P$, respectively.
\end{proposition}

Let $\sigma$ be a Pierce admissible word of length $k$, and let $y\in I_P(\sigma)$. Since the first Pierce expansion digit of each $z\in\I$ is $\lfloor1/z\rfloor$, it follows from \eqref{eq:expansions} that
\begin{equation}\label{eq:pierce-tail}
 y=r(\sigma)+\frac{(-1)^k z}{\sigma_1\cdots\sigma_k}
 \qquad\text{and}\qquad
 \frac1{p_{k+1}(y)+1}<z<\frac1{p_{k+1}(y)},
\end{equation}
where $z$ is the element of $\I$ whose Pierce expansion digit sequence is $(p_{k+n}(y))_{n\ge1}$.

Since the cylinder sets form a base of $\Sigma_\star$ and $f_\star$ is a homeomorphism with inverse $\varphi_\star$ by Proposition~\ref{prop:coding}, the fundamental intervals $I_\star(a)$ form a base of $\I$ for each $\star\in\{E,P\}$. In particular, every non-empty open subset of $\I$ contains some $I_E(a)$ and some $I_P(\sigma)$. Moreover, $\I$, $\E$, and $\Pp$ are Polish spaces without isolated points.

\subsection{Modified Engel expansions and Moroz's maps}\label{sec:moroz}

Following Moroz \cite{Mor24,Mor}, every $x\in(0,1]$ has a unique \emph{modified Engel expansion}
\begin{equation}\label{eq:modified}
 x=\sum_{n=1}^{\infty}
 \frac{1}{(m_1-1)\cdots(m_{n-1}-1)m_n},
 \qquad 2\le m_1<m_2<\cdots .
\end{equation}
For the classical Engel expansion, Moroz uses the non-terminating convention, so that rational numbers correspond to eventually constant Engel expansion digit sequences. Accordingly, we write $e_n(x)$ for the $n$th digit of this expansion of $x\in(0,1]$; for $x\in\I$, this agrees with \eqref{eq:expansions}. Also, the Pierce expansion digits in \cite{Mor} are one larger than the traditional ones. In the notation of \cite{Mor}, \eqref{eq:modified} is the positive Perron expansion, and the Pierce expansion is the alternating Perron expansion, associated with the sequence $\sP$ of functions consisting of the constant $1$ and the maps $(x_1,\dotsc,x_n)\mapsto x_n$, $n\in\N$ (see \cite{Mor24,Mor25}). Accordingly, for this $\sP$, the maps $\cT\colon(0,1]\to(0,1]$ and $\cF\colon(0,1]\to\I$ in \cite{Mor} are given by
\[
 m_n(\cT(x))\coloneqq e_n(x)+n-1
 \qquad\text{and}\qquad
 p_n(\cF(x))\coloneqq m_n(x)-1
\]
for all $x\in(0,1]$ and $n\in\N$. Both maps are bijections, $\cT$ is continuous and strictly increasing \cite[Section~4.2]{Mor}, and $\cF=\cF_{\sP}$ in the notation of \cite[Section~4.1]{Mor}.

Let
\[
 \cB\coloneqq\{x\in(0,1]:m_{n+1}(x)=m_n(x)+1
       \text{ for all sufficiently large }n\}.
\]
A finite sequence $u \coloneqq (u_1,\dotsc,u_k)$ of integers with $2\le u_1<\dotsb<u_k$, where $k\in\N_0$, is called a \emph{modified Engel admissible word of length $k$}. For the empty word, we put $u_0\coloneqq1$. For such $u$, put
\[
 R(u)\coloneqq\prod_{j=1}^{k}\frac1{u_j-1}
 \qquad\text{and}\qquad
 s(u)\coloneqq\sum_{n=1}^{k}\frac{R(u_1,\ldots,u_{n-1})}{u_n}.
\]
Let $J(u)\coloneqq\{x\in(0,1]:m_j(x)=u_j\text{ for all }1\le j\le k\}$ be the modified Engel fundamental interval associated with $u$ (called a cylinder in \cite{Mor}). For $x\in J(u)$, it follows from \eqref{eq:modified} that $x=s(u)+R(u)\,V$, where $V$ is the value of the tail $(m_{k+1}(x),m_{k+2}(x),\ldots)$ in \eqref{eq:modified}. Note that a tail starting with $j$ is dominated term by term by the consecutive tail $j,j+1,\ldots$, whose value telescopes to $1/(j-1)$. Since $m_{k+1}(x)\ge u_k+1$, we obtain $1/m_{k+1}(x)<V\le1/u_k$, with equality exactly for the consecutive tail $u_k+1,u_k+2,\ldots$. By \cite[Proposition~2.5]{Mor}, $J(u)$ is a left-open, right-closed interval of length $R(u)/u_k$, and so
\begin{align}\label{eq:modified-cylinder}
J(u)
=\Bigl(s(u),\,s(u)+\frac{R(u)}{u_k}\Bigr], \qquad
J(u,j)
=\Bigl(s(u)+\frac{R(u)}{j},\,
        s(u)+\frac{R(u)}{j-1}\Bigr],
\end{align}
for each $j>u_k$, where the second formula is the first one applied to $(u,j)$. Thus $\cB$ is precisely the set of right endpoints of the modified Engel fundamental intervals. These endpoints are rational, and so $\cB\subseteq\Q\cap(0,1]$. Moreover, $\cB$ is countable, and it is dense in $(0,1]$ since $|J(u)|=R(u)/u_k\le1/k!$.

\section{Main results}\label{sec:main}

In this section, we state and prove the main results of this paper.

\subsection{Baire-category transfer}\label{sec:transfer}

We first construct the correspondence between Engel and Pierce expansion digit sequences. Put 
\[
\Pup\coloneqq\{\pi\in\Pp:\pi_n-n\to\infty\text{ as }n\to\infty\},
\]
and define the maps $\Phi\colon\E\to\Pp$ and $\Psi\colon\Pup\to\E$ by
\[
 \Phi(\eta)_n\coloneqq\eta_n+n-2
 \qquad\text{and}\qquad
 \Psi(\pi)_n\coloneqq\pi_n-n+2
 \qquad(n\in\N).
\]
For an Engel admissible word $a$ of length $k$, put $T(a)\coloneqq(a_j+j-2)_{j=1}^{k}$. Since $a_1-1\ge1$ and $(a_{j+1}+j-1)-(a_j+j-2)=a_{j+1}-a_j+1\ge1$, the word $T(a)$ is Pierce admissible, and we call it the Pierce admissible word \emph{paired} with $a$.

\begin{proposition}\label{prop:symbolic}
The following hold.
\begin{enumerate}[label=\textup{(\roman*)},ref=\textup{(\roman*)},leftmargin=*]
\item\label{prop:symbolic:i}
The map $\Phi$ is a homeomorphism of $\E$ onto $\Pup$ with inverse $\Psi$. For each $k\in\N_0$, the map $T$ is a bijection from the set of Engel admissible words of length $k$ onto the set of Pierce admissible words of length $k$, and
\begin{equation}\label{eq:cylinders}
 \Phi([a]_E)=[T(a)]_P\cap\Pup
\end{equation}
for each Engel admissible word $a$.

\item\label{prop:symbolic:ii}
For each $\pi\in\Pp$, the sequence $(\pi_n-n)_{n\ge1}$ is non-decreasing, and
\begin{equation}\label{eq:exceptional-tail}
 \Pp\setminus\Pup
 =\{\pi\in\Pp:\pi_{n+1}=\pi_n+1
   \text{ for all sufficiently large }n\}.
\end{equation}
Moreover, this set is countable and dense in $\Pp$.
\end{enumerate}
\end{proposition}

\begin{proof}
\ref{prop:symbolic:i}
If $\eta\in\E$, then
$\Phi(\eta)_1=\eta_1-1\ge1$,
$\Phi(\eta)_{n+1}-\Phi(\eta)_n=\eta_{n+1}-\eta_n+1\ge1$, and
$\Phi(\eta)_n-n=\eta_n-2\to\infty$ as $n\to\infty$.
Conversely, if $\pi\in\Pup$, then
$\Psi(\pi)_1=\pi_1+1\ge2$,
$\Psi(\pi)_{n+1}-\Psi(\pi)_n=\pi_{n+1}-\pi_n-1\ge0$, and
$\Psi(\pi)_n\to\infty$ as $n\to\infty$.
The same computations apply to admissible words. The maps $\Phi$ and $\Psi$ are mutually inverse, and each coordinate of the output depends only on one coordinate of the input, so that both maps are continuous. Identity \eqref{eq:cylinders} holds since $\Phi(\eta)_j=T(a)_j$ for all $j\le k$ if and only if $\eta_j=a_j$ for all $j\le k$.

\ref{prop:symbolic:ii}
The monotonicity follows from $\pi_{n+1}\ge\pi_n+1$. A non-decreasing integer sequence which does not tend to infinity is eventually constant, and $(\pi_n-n)_{n\ge1}$ is eventually constant if and only if $\pi$ has an eventually consecutive tail. Such a $\pi$ is determined by a Pierce admissible word together with the index from which the increments are equal to $1$, and so the set in \eqref{eq:exceptional-tail} is countable. Every cylinder set $[\sigma]_P$ contains the sequence
\[
(\sigma_1,\ldots,\sigma_k,\sigma_k+1,\sigma_k+2,\ldots),
\]
where $\sigma_0\coloneqq0$, which proves the denseness.
\end{proof}

Define
\begin{equation}\label{eq:theta}
 C\coloneqq\varphi_P(\Pp\setminus\Pup),\qquad \Is\coloneqq\I\setminus C,\qquad
 \Theta\coloneqq\varphi_P\circ\Phi\circ f_E\colon\I\to\I .
\end{equation}
Then $p_n(\Theta(x))=e_n(x)+n-2$ for all $x\in\I$ and $n\in\N$. By Proposition~\ref{prop:moroz-composition} below, $\Theta$ coincides with the restriction to $\I$ of the composition $\cF\circ\cT$ of Moroz's maps. By \eqref{eq:exceptional-tail}, the set $C$ consists of the irrational numbers whose Pierce expansion digits are eventually consecutive, and $C$ is countable and dense in $\I$ by Propositions~\ref{prop:coding} and~\ref{prop:symbolic}\ref{prop:symbolic:ii}. Since $\I$ has no isolated points, $C$ is meager, and $\Is$ is a dense $G_\delta$ subset of $\I$.

\begin{theorem}\label{thm:transfer}
The following hold.
\begin{enumerate}[label=\textup{(\roman*)},ref=\textup{(\roman*)},leftmargin=*]

\item\label{thm:transfer:i}
The map $\Theta$ is a homeomorphism of $\I$ onto $\Is$ with inverse
$\varphi_E\circ\Psi\circ f_P|_{\Is}$, and
\begin{equation}\label{eq:real-cylinders}
 \Theta(I_E(a))=I_P(T(a))\setminus C
\end{equation}
for each Engel admissible word $a$.

\item\label{thm:transfer:ii}
Let $A\subseteq\I$. Then $A$ is dense in $\I$ if and only if
$\Theta(A)$ is dense in $\I$. The same holds if ``dense'' is replaced by
``meager'', ``comeager'', or ``nonmeager''.

\item\label{thm:transfer:iii}
Let $A\subseteq\I$, and let $\alpha$ be a countable ordinal. If $\alpha\ge2$,
then $A\in\PiB{\alpha}(\I)$ if and only if
$\Theta(A)\in\PiB{\alpha}(\I)$. If $\alpha\ge3$, then
$A\in\SigmaB{\alpha}(\I)$ if and only if
$\Theta(A)\in\SigmaB{\alpha}(\I)$. In particular, $A$ is Borel if and only
if $\Theta(A)$ is Borel.

\item\label{thm:transfer:iv}
The ranges of $\alpha$ in \ref{thm:transfer:iii} are optimal. More precisely,
$\Theta(\I)=\Is$ belongs to $\PiB2(\I)$ but not to $\SigmaB2(\I)$. Thus
$\Theta$ maps the clopen set $\I$ onto a set which is neither open, nor closed,
nor $F_\sigma$ in $\I$.

\end{enumerate}
\end{theorem}

\begin{proof}
\ref{thm:transfer:i}
The first assertion follows from Propositions~\ref{prop:coding}
and~\ref{prop:symbolic}\ref{prop:symbolic:i}. Since $\varphi_P$ is a
bijection, it follows from \eqref{eq:cylinders} that
\[
\Theta(I_E(a))
=\varphi_P([T(a)]_P\cap\Pup)
=I_P(T(a))\setminus C.
\]

\ref{thm:transfer:ii}
Since $\Theta\colon\I\to\Is$ is a homeomorphism, $A$ has each of the four
properties in $\I$ if and only if $\Theta(A)$ has the same property in $\Is$.
It remains to show that, for $B\subseteq\Is$, each property in $\Is$ is
equivalent to the same property in $\I$. For denseness, this holds since
$\Is$ is dense in $\I$. Note that the closure of $B$
in $\Is$ is $\Is\cap\overline{B}$, where $\overline B$ denotes the closure
of $B$ in $\I$. If $O$ is a non-empty open subset of $\I$ with
$O\subseteq\overline B$, then $O\cap\Is$ is non-empty and contained in
$\Is\cap\overline B$. Conversely, if
$O\cap\Is\subseteq\overline B$ and $O\cap\Is\ne\varnothing$, then
\[
O\subseteq\overline{O\cap\Is}\subseteq\overline B
\]
by the denseness of $\Is$. Hence $B$ is nowhere dense in $\Is$ if and only
if it is nowhere dense in $\I$. By taking countable unions, and intersecting the nowhere dense sets with $\Is$ if necessary, we see that
$B$ is meager in $\Is$ if and only if it is meager in $\I$. Since
$\I\setminus B=C\cup(\Is\setminus B)$ and $C$ is meager in $\I$, it follows
that $B$ is comeager in $\Is$ if and only if it is comeager in $\I$.
Finally, nonmeagerness is the negation of meagerness.

\ref{thm:transfer:iii}
Let $\Gamma(X)$ denote either $\PiB{\alpha}(X)$ with $\alpha\ge2$ or
$\SigmaB{\alpha}(X)$ with $\alpha\ge3$. For $B\subseteq\Is$, we have
$B\in\Gamma(\Is)$ if and only if $B=\Is\cap B'$ for some
$B'\in\Gamma(\I)$ (see \cite[Section~11.B]{Ke}). Since
$\Is\in\PiB2(\I)\subseteq\Gamma(\I)$ and $\Gamma(\I)$ is closed under finite
intersections, we have
\[
\Gamma(\Is)=\{B\subseteq\Is:B\in\Gamma(\I)\}.
\]
Since $\Theta$ is a homeomorphism of $\I$ onto $\Is$, it follows that
$A\in\Gamma(\I)$ if and only if $\Theta(A)\in\Gamma(\Is)$, which in turn
holds if and only if $\Theta(A)\in\Gamma(\I)$.

\ref{thm:transfer:iv}
We already know that $\Is\in\PiB2(\I)$. If $\Is$ were $F_\sigma$ in $\I$,
then $C$ would be a dense $G_\delta$ subset of $\I$, and hence comeager in
$\I$. But $C$ is meager in $\I$, and $\I$ is a Baire space. This is a
contradiction.
\end{proof}

For the classes $\SigmaB{1}$, $\PiB{1}$, and $\SigmaB{2}$, which are excluded in Theorem~\ref{thm:transfer}\ref{thm:transfer:iii}, only one direction survives. Indeed, since $\Theta\colon\I\to\I$ is continuous and injective, $\Theta(A)\in\Gamma(\I)$ implies $A=\Theta^{-1}(\Theta(A))\in\Gamma(\I)$ for every Borel class $\Gamma$.

\begin{remark}\label{rem:predicates}
For every $\mathcal Q\subseteq\N^{\N}$, we have
\[
 \Theta^{-1}\bigl(\{y\in\Is:(p_n(y))_{n\ge1}\in\mathcal Q\}\bigr)
 =\{x\in\I:(e_n(x)+n-2)_{n\ge1}\in\mathcal Q\}.
\]
Combining this with Theorem~\ref{thm:transfer}, we see that the Baire category of a set of irrational numbers defined by a condition on the Pierce expansion digits can be determined on the Engel side, and vice versa. Since the countable set $C$ is meager, it never affects meagerness, comeagerness, or nonmeagerness; it needs to be examined separately only when denseness or membership in $\SigmaB{\alpha}$ or $\PiB{\alpha}$ with $\alpha\le2$ is concerned.
\end{remark}

\subsection{Continuity properties of Moroz's map}\label{subsec:moroz-main}

We first identify $\Theta$ with the composition of Moroz's maps.

\begin{proposition}\label{prop:moroz-composition}
We have
\begin{gather*}
 \cF(\cB)=C
 \qquad\text{and}\qquad
 \cF((0,1]\setminus\cB)=\Is,\\
 \cT(\Q\cap(0,1])=\cB
 \qquad\text{and}\qquad
 \cT(\I)=(0,1]\setminus\cB.
\end{gather*}
Moreover, $\cF\circ\cT|_{\I}=\Theta$.
\end{proposition}

\begin{proof}
Note that the Pierce expansion digits of $\cF(x)$ are $m_n(x)-1$, and so they are eventually consecutive if and only if the modified Engel expansion digits of $x$ are. Similarly, $(e_n(x)+n-1)_{n\ge1}$ is eventually consecutive if and only if $(e_n(x))_{n\ge1}$ is eventually constant, which, under Moroz's convention, is equivalent to $x\in\Q$. The last identity follows from $p_n(\cF(\cT(x)))=e_n(x)+n-2$.
\end{proof}

Recall that the exceptional set $IS^{\sP^-}$ in Moroz's conjecture \cite[Section~4.1]{Mor} is the countable set of numbers in $(0,1]$ without an alternating Perron representation (see \cite[Section~2]{Mor}). For our pair of expansions, the range of $\cF$ is $\I$, and so $IS^{\sP^-}=\Q\cap(0,1]$. Moroz \cite{Mor26} obtained necessary and sufficient conditions for the convergence of sequences in terms of positive and alternating Perron expansions, which make it possible to verify the continuity of functions defined in terms of these expansions. For the present pair of expansions, we give a direct proof. The following theorem proves the conjectured nowhere monotonicity and continuity outside $IS^{\sP^-}$, and shows that the set of discontinuities is exactly $\cB\setminus\{1\}\subseteq IS^{\sP^-}$ and that every discontinuity is a jump.

\begin{theorem}\label{thm:moroz}
The following hold.
\begin{enumerate}[label=\textup{(\roman*)},ref=\textup{(\roman*)},leftmargin=*]

\item\label{thm:moroz:i}
The map $\cF$ is left-continuous on $(0,1]$ and is continuous at every point of $\bigl((0,1]\setminus\cB\bigr)\cup\{1\}$.

\item\label{thm:moroz:ii}
If $b\in\cB\setminus\{1\}$, then $\lim_{t\downarrow b}\cF(t)$ exists in $\R$ and is a rational endpoint of a Pierce fundamental interval. Since $\cF(b)$ is irrational, $\cF$ has a jump discontinuity at $b$.

\item\label{thm:moroz:iii}
The map $\cF$ is not monotone on any non-degenerate subinterval of $(0,1]$.

\end{enumerate}
Consequently, the set of discontinuities of $\cF$ is the countable dense set $\cB\setminus\{1\}\subseteq\Q\cap(0,1)$. In particular, $\cF$ is continuous at every irrational point of $(0,1)$, and at $1$ with respect to the relative topology.
\end{theorem}

\begin{proof}
For a modified Engel admissible word $u$ of length $k$, write
$u-1\coloneqq(u_1-1,\ldots,u_k-1)$. Then
$\cF(J(u))\subseteq I_P(u-1)$. Moreover, if $k\ge1$, then $|I_P(u-1)|<1/k!$ by
\cite[Proposition~3.5]{AhnJFG}, since $u_j-1\ge j$ for each $j$.

\ref{thm:moroz:i}
Let $x\in(0,1]$ and $u\coloneqq(m_1(x),\ldots,m_k(x))$. By
\eqref{eq:modified-cylinder}, $J(u)$ contains an interval immediately to the left of $x$. Hence $|\cF(t)-\cF(x)|<1/k!$ for all $t<x$ sufficiently close to $x$, which proves the left-continuity. If $x\notin\cB$, then $x$ is an interior point of $J(u)$ for every such $u$, and so the same argument applies on both sides. At $x=1$, continuity on $(0,1]$ is the same as left-continuity.

\ref{thm:moroz:ii}
Let $b\in\cB\setminus\{1\}$, and let $k$ be the smallest positive integer such that $m_{n+1}(b)=m_n(b)+1$ for all $n\ge k$. Write $m_n\coloneqq m_n(b)$, and put $u\coloneqq(m_1,\ldots,m_{k-1})$ and $u'\coloneqq(m_1,\ldots,m_{k-1},m_k-1)$. Then $m_k\ge m_{k-1}+2$, where $m_0\coloneqq1$. Indeed, for $k\ge2$, this follows from the minimality of $k$, while for $k=1$, it follows from $b\ne1$, since the digits $2,3,4,\ldots$ represent~$1$. Thus $u'$ is a modified Engel admissible word. By \eqref{eq:modified-cylinder}, $b$ is the right endpoint of $J(u,m_k)$, and so
\[
b=s(u)+\frac{R(u)}{m_k-1}=s(u'),
\]
and $J(u')$ lies immediately to the right of $b$. If $t\in J(u')$, then
\[
 t-b=R(u')\,V(m_{k+1}(t),m_{k+2}(t),\ldots)
      >\frac{R(u')}{m_{k+1}(t)},
\]
where $V$ denotes the value of the modified Engel tail. Thus
$m_{k+1}(t)\to\infty$ as $t\downarrow b$. Put $\sigma\coloneqq u'-1$. Then
\eqref{eq:pierce-tail} gives
\[
 |\cF(t)-r(\sigma)|
 <\frac{1}{\sigma_1\cdots\sigma_k\,(m_{k+1}(t)-1)}
 \to0
 \quad\text{as }t\downarrow b.
\]
Hence the right-hand limit of $\cF$ at $b$ is the rational number $r(\sigma)$, whereas $\cF(b)\in\I$.

\ref{thm:moroz:iii}
Let $u$ be a modified Engel admissible word of length $\ell$, and let $u_\ell<j<j'$. By \eqref{eq:modified-cylinder}, $J(u,j')$ lies to the left of $J(u,j)$, and so $x'<x$ whenever $x\in J(u,j)$ and $x'\in J(u,j')$. For such $x$ and $x'$, the points $\cF(x)$ and $\cF(x')$ lie in $I_P(u-1,j-1)$ and $I_P(u-1,j'-1)$, respectively. We apply \eqref{eq:pierce-tail} with $k=\ell$ and $\sigma=u-1$. Then the corresponding tail values satisfy
\[
 z\in\Bigl(\frac1j,\frac1{j-1}\Bigr)
 \qquad\text{and}\qquad
 z'\in\Bigl(\frac1{j'},\frac1{j'-1}\Bigr),
\]
and $1/(j'-1)\le1/j$, so that $z>z'$. Hence
\[
 (-1)^\ell\bigl(\cF(x)-\cF(x')\bigr)
 =\frac{z-z'}{(u_1-1)\cdots(u_\ell-1)}>0 .
\]
Thus, inside $J(u)$, the map $\cF$ preserves the order of two points taken from distinct subintervals of the form $J(u,j)$ when $\ell$ is even, and reverses it when $\ell$ is odd.

Now, let $U\subseteq(0,1]$ be a non-degenerate interval. Pick a point $x$ in the interior of $U$. Since $J(m_1(x),\ldots,m_k(x))$ contains $x$ and has length at most $1/k!$, we have $J(u)\subseteq U$ for some modified Engel admissible word $u$ of some length $k$. Then $U$ also contains $J(u,j)$ for each $j>u_k$. Applying the preceding paragraph inside $J(u)$ and inside $J(u,j)$, whose associated words have lengths $k$ and $k+1$ of opposite parity, we find $t_1<t_2$ and $t_3<t_4$ in $U$ with $\cF(t_1)<\cF(t_2)$ and $\cF(t_3)>\cF(t_4)$. Therefore, $\cF$ is not monotone on $U$. This completes the proof.
\end{proof}

\begin{remark}\label{rem:B}
By Proposition~\ref{prop:moroz-composition}, we have $\cB=\cT(\Q\cap(0,1])$. It remains open whether $\cB=\Q\cap(0,1]$, or equivalently, whether every point of $IS^{\sP^-}\setminus\{1\}$ is a discontinuity of $\cF$.
\end{remark}

\begin{remark}\label{rem:general-P}
The proof of Theorem~\ref{thm:moroz}\ref{thm:moroz:i} uses only the following three facts: the cylinders of the positive expansion are left-open, right-closed intervals; $\cF_{\sP}$ maps each of them onto the cylinder of the alternating expansion with the same base; and the two cylinders have equal diameters, which tend to $0$ along nested sequences \cite[Proposition~2.5]{Mor}. Hence, for every $\sP$, the map $\cF_{\sP}$ is left-continuous on $(0,1]$ and is continuous away from the right endpoints of the positive $\sP$-cylinders. The nowhere monotonicity and the analysis of the jumps are established here only for the pair of modified Engel and Pierce expansions.
\end{remark}

\begin{remark}\label{rem:dimension}
By Proposition~\ref{prop:moroz-composition} and \cite[Theorem~4.1 and Corollary~4.4]{Mor}, we have $\dimH\Theta(A)\le\dimH A$ for every $A\subseteq\I$. Moreover, the equality holds if, for each $x\in A$, we have $e_n(x)\ge\psi(n)$ for all sufficiently large $n$, where $\psi\colon\N\to(0,\infty)$ satisfies $\sum_{n\ge1} n/\psi(n)<\infty$ (see \cite[Theorems~4.1 and~4.5]{Mor}). On the other hand, the Baire-category transfer in Theorem~\ref{thm:transfer} requires no such growth condition.
\end{remark}

\subsection{Comeager divergence of digit series}\label{sec:series}

We now apply the Baire-category transfer to series whose terms depend on both the position and the value of the digits. For functions $v_n\colon\{2,3,\ldots\}\to[0,\infty)$ and $w_n\colon\N\to[0,\infty)$, $n\in\N$, put
\begin{align*}
 \cD^E_v\coloneqq\left\{x\in\I:\sum_{n=1}^{\infty}v_n(e_n(x))=\infty\right\},\qquad
 \cD^P_w\coloneqq\left\{y\in\I:\sum_{n=1}^{\infty}w_n(p_n(y))=\infty\right\}.
\end{align*}

\begin{theorem}\label{thm:engel-series}
Let $T_v\coloneqq\{t\ge2:\sum_{n\ge1}v_n(t)=\infty\}$. Then the following hold.
\begin{enumerate}[label=\textup{(\roman*)},ref=\textup{(\roman*)},leftmargin=*]

\item\label{thm:engel-series:i}
The set $\cD^E_v$ is a $G_\delta$ subset of $\I$.

\item\label{thm:engel-series:ii}
If $T_v$ is unbounded, then $\cD^E_v$ is a dense $G_\delta$ subset of $\I$.

\item\label{thm:engel-series:iii}
Suppose that each $v_n$ is non-increasing. Then $\cD^E_v$ is dense in $\I$ if and only if it is comeager in $\I$, which in turn holds if and only if $T_v=\{2,3,\ldots\}$. Moreover, if $t_0\notin T_v$, then $\cD^E_v$ is disjoint from the fundamental interval
\[
I_E(t_0)=\{x\in\I:e_1(x)=t_0\}.
\]

\item\label{thm:engel-series:iv}
If $v_n(t)\to0$ as $t\to\infty$ for each $n\in\N$, then $\I\setminus\cD^E_v$ is dense in $\I$.

\end{enumerate}
\end{theorem}

\begin{proof}
\ref{thm:engel-series:i}
The partial sum
\[
S_N(x)\coloneqq\sum_{n=1}^{N}v_n(e_n(x))
\]
depends only on the first $N$ digits of $x$, and so it is locally constant, and hence continuous, on $\I$. Since the terms are non-negative, we have
\[
 \cD^E_v=\bigcap_{M\in\N}U_M,
 \qquad\text{where}\quad
 U_M\coloneqq\bigcup_{N\in\N}\{x\in\I:S_N(x)>M\},
\]
and each $U_M$ is open in $\I$.

\ref{thm:engel-series:ii}
Let $M\in\N$, and let $O\subseteq\I$ be a non-empty open set. Since the fundamental intervals form a base of $\I$, we have $O\supseteq I_E(a)$ for some Engel admissible word $a \coloneqq (a_1,\dotsc,a_K)$, where we put $a_0\coloneqq2$. Choose $t\in T_v$ with $t\ge a_K$. Since the tail $\sum_{n>K}v_n(t)$ still diverges, there exists $L\in\N$ such that
\[
\sum_{n=K+1}^{K+L}v_n(t)>M.
\]
The finite sequence $a^{(L)}\coloneqq(a_1,\dotsc,a_K,t,\dotsc,t)$, with $L$ copies of $t$, is an Engel admissible word, and we have $S_{K+L}>M$ on the non-empty set $I_E(a^{(L)})\subseteq O$. Hence each $U_M$ is dense in $\I$, and \ref{thm:engel-series:ii} follows from the Baire category theorem.

\ref{thm:engel-series:iii}
Suppose that each $v_n$ is non-increasing. If $t\in T_v$ and $2\le t'\le t$, then $t'\in T_v$. Hence $T_v$ is unbounded if and only if $T_v=\{2,3,\ldots\}$. If $t_0\notin T_v$, then every $x\in I_E(t_0)$ satisfies $e_n(x)\ge t_0$ for all $n\in\N$, and hence
\[
\sum_{n\ge1}v_n(e_n(x))
\le
\sum_{n\ge1}v_n(t_0)
<\infty.
\]
Combining this with \ref{thm:engel-series:ii} and the Baire category theorem, we obtain \ref{thm:engel-series:iii}.

\ref{thm:engel-series:iv}
Let $I_E(a)$ be as in the proof of \ref{thm:engel-series:ii}. Put $t_K\coloneqq a_K$, and choose inductively, for each $n>K$, an integer
\[
t_n\ge\max\{t_{n-1},n\}
\]
with $v_n(t_n)\le2^{-n}$. Then the sequence
\[
(a_1,\ldots,a_K,t_{K+1},t_{K+2},\ldots)
\]
belongs to $\E$, and its image under $\varphi_E$ lies in $I_E(a)\setminus\cD^E_v$. This completes the proof.
\end{proof}

If $v_n=g$ for all $n\in\N$, where $g\colon\{2,3,\ldots\}\to[0,\infty)$, then $T_v=\{t\ge2:g(t)>0\}$, and so Theorem~\ref{thm:engel-series}\ref{thm:engel-series:ii} recovers \cite[Theorem~4.1]{AhnEngel}. Thus Theorem~\ref{thm:engel-series} extends \cite[Theorem~4.1]{AhnEngel} by allowing the functions $v_n$ to depend on $n$, which is needed in the proof of Theorem~\ref{thm:pierce-series}.

\begin{theorem}\label{thm:pierce-series}
Let $\Lambda_w\coloneqq\{c\in\N_0:\sum_{n\ge1}w_n(n+c)=\infty\}$. Then the following hold.
\begin{enumerate}[label=\textup{(\roman*)},ref=\textup{(\roman*)},leftmargin=*]

\item\label{thm:pierce-series:i}
The set $\cD^P_w$ is a $G_\delta$ subset of $\I$.

\item\label{thm:pierce-series:ii}
If $\Lambda_w$ is unbounded, then $\cD^P_w$ is a dense $G_\delta$ subset of $\I$.

\item\label{thm:pierce-series:iii}
Suppose that each $w_n$ is non-increasing. Then $\cD^P_w$ is dense in $\I$ if and only if it is comeager in $\I$, which in turn holds if and only if $\Lambda_w=\N_0$. Moreover, if $c_0\notin\Lambda_w$, then $\cD^P_w$ is disjoint from the fundamental interval
\[
I_P(c_0+1)=\{y\in\I:p_1(y)=c_0+1\}.
\]

\item\label{thm:pierce-series:iv}
If $w_n(t)\to0$ as $t\to\infty$ for each $n\in\N$, then $\I\setminus\cD^P_w$ is dense in $\I$.

\end{enumerate}
\end{theorem}

\begin{proof}
Define $v_n(t)\coloneqq w_n(t+n-2)$ for $t\ge2$. Then
\[
\sum_{n\ge1}w_n(p_n(\Theta(x)))
=
\sum_{n\ge1}v_n(e_n(x))
\]
for each $x\in\I$ by \eqref{eq:digit-shift}, and hence
\begin{equation}\label{eq:series-transfer}
 \Theta(\cD^E_v)=\cD^P_w\cap\Is,\qquad
 \Theta(\I\setminus\cD^E_v)=\Is\setminus\cD^P_w,
 \qquad
 T_v=\Lambda_w+2 .
\end{equation}

\ref{thm:pierce-series:i}
This is proved in the same way as Theorem~\ref{thm:engel-series}\ref{thm:engel-series:i}.

\ref{thm:pierce-series:ii}
If $\Lambda_w$ is unbounded, then so is $T_v$. By
Theorem~\ref{thm:engel-series}\ref{thm:engel-series:ii},
$\cD^E_v$ is a dense $G_\delta$ subset of $\I$, and so is
$\Theta(\cD^E_v)$ by Theorem~\ref{thm:transfer}. By
\eqref{eq:series-transfer}, $\cD^P_w$ contains a dense $G_\delta$ subset of
$\I$, and hence it is itself a dense $G_\delta$ subset of $\I$ by
\ref{thm:pierce-series:i}.

\ref{thm:pierce-series:iii}
Suppose that each $w_n$ is non-increasing. Then so is each $v_n$, and
$\Lambda_w$ is an initial segment of $\N_0$. Let $c_0\notin\Lambda_w$.
For each $y\in I_P(c_0+1)$, we have $p_n(y)\ge n+c_0$, and hence
\[
w_n(p_n(y))\le w_n(n+c_0)
\]
for all $n\in\N$, and so the series converges. Hence $\cD^P_w\cap I_P(c_0+1)=\varnothing$.
Combining this with \ref{thm:pierce-series:ii} and the Baire category theorem,
we obtain \ref{thm:pierce-series:iii}.

\ref{thm:pierce-series:iv}
In this case, $v_n(t)\to0$ as $t\to\infty$ for each $n\in\N$. By
Theorem~\ref{thm:engel-series}\ref{thm:engel-series:iv} and
Theorem~\ref{thm:transfer}, the set
\[
\Theta(\I\setminus\cD^E_v)=\Is\setminus\cD^P_w
\]
is dense in $\I$, and so is $\I\setminus\cD^P_w$. This completes the proof.
\end{proof}

The mechanism behind the shift from $t$ to $n+c$ can be seen from the proof. The slowest possible growth of Engel expansion digits is given by long constant blocks $e_n=t$, and, under \eqref{eq:digit-shift}, they correspond exactly to long consecutive blocks $p_n=n+t-2$ of Pierce expansion digits. Similarly, by \eqref{eq:real-cylinders}, we have $I_P(c_0+1)=\Theta(I_E(c_0+2))\cup(C\cap I_P(c_0+1))$ for each $c_0\in\N_0$. Thus the fundamental interval $I_P(c_0+1)$ in Theorem~\ref{thm:pierce-series}\ref{thm:pierce-series:iii} corresponds, up to points of $C$, to the fundamental interval $I_E(t_0)$ with $t_0=c_0+2$ in Theorem~\ref{thm:engel-series}\ref{thm:engel-series:iii}.

\section{Consequences}\label{sec:consequences}

In this section, we present some consequences of the main results for series of Pierce expansion digits and for convergence exponents.

\subsection{Series of Pierce expansion digits}\label{subsec:g}

By applying Theorem~\ref{thm:pierce-series} to functions which do not depend on the position $n$, we obtain the following corollary, in which no monotonicity assumption is needed.

\begin{corollary}\label{cor:g}
Let $g\colon\N\to[0,\infty)$, and put
\[
\cD_g\coloneqq\left\{y\in\I:\sum_{n\ge1}g(p_n(y))=\infty\right\}.
\]
Then the following hold.
\begin{enumerate}[label=\textup{(\roman*)},ref=\textup{(\roman*)},leftmargin=*]

\item\label{cor:g:i}
If $\sum_{n\ge1}g(n)=\infty$, then $\cD_g$ is a dense $G_\delta$ subset of $\I$ and has cardinality $\cont$. If, in addition, $g(t)\to0$ as $t\to\infty$, then $\I\setminus\cD_g$ is dense and meager in $\I$.

\item\label{cor:g:ii}
If $\sum_{n\ge1}g(n)<\infty$, then $\cD_g=\varnothing$.

\end{enumerate}
In particular, the set
\[
\left\{y\in\I:\sum_{n\ge1}p_n(y)^{-s}=\infty\right\}
\]
is a dense $G_\delta$ subset of $\I$ and has cardinality $\cont$ for each
$s\in(0,1]$, and it is empty for each $s>1$.
\end{corollary}

\begin{proof}
\ref{cor:g:i}
Put $w_n\coloneqq g$ for each $n\in\N$. Then each shifted series
\[
\sum_{n\ge1}g(n+c)
\]
is a tail of $\sum_{n\ge1}g(n)$, and so $\Lambda_w=\N_0$. Hence $\cD_g$ is a dense $G_\delta$ subset of $\I$ by
Theorem~\ref{thm:pierce-series}\ref{thm:pierce-series:ii}. Note that a dense $G_\delta$ subset of $\I$ is a Polish space \cite[Theorem~3.11]{Ke} without isolated points, and therefore has cardinality $\cont$ \cite[Section~6.A]{Ke}. If, in addition, $g(t)\to0$ as $t\to\infty$, then $\I\setminus\cD_g$ is dense in $\I$ by
Theorem~\ref{thm:pierce-series}\ref{thm:pierce-series:iv}, and it is meager in $\I$ as the complement of a dense $G_\delta$ subset.

\ref{cor:g:ii}
For each $y\in\I$, the series $\sum_{n\ge1}g(p_n(y))$ is a subseries of the convergent series $\sum_{m\ge1}g(m)$ of non-negative terms, since the Pierce expansion digits of $y$ are distinct.
\end{proof}

For $s\in(0,1]$, the last assertion of Corollary~\ref{cor:g} recovers \cite[Theorem~1.13]{AhnIJNT}, including the critical case $s=1$. Note that this cannot be obtained from convergence exponents alone, since $\lambda_P\le1$ does not decide the convergence at $s=1$. Similarly, the series $\sum_{n\ge1}\bigl((p_n+1)\log(p_n+1)\bigr)^{-1}$ diverges on a dense $G_\delta$ subset of $\I$, whereas the series $\sum_{n\ge1}\bigl((p_n+1)\log^2(p_n+1)\bigr)^{-1}$ converges at every point of $\I$. By contrast, on the Engel side, $\sum_{n\ge1}g(e_n)$ diverges on a comeager subset of $\I$ as soon as $g>0$ on an unbounded set \cite[Theorem~4.1]{AhnEngel}.

\subsection{Convergence exponents}\label{sec:exponents}

We now relate the convergence exponents of Engel and Pierce expansion digit sequences. Recall that, for a non-decreasing sequence $d \coloneqq (d_n)_{n\ge1}$ of positive integers with $d_n\to\infty$ as $n\to\infty$, the convergence exponent of $d$ is defined by
\[
 \lambda(d)\coloneqq\inf \left\{s>0:\sum_{n=1}^{\infty}d_n^{-s}<\infty \right\},
\]
and it is well known that
\begin{equation}\label{eq:exponent}
 \lambda(d)=\limsup_{n\to\infty}\frac{\log n}{\log d_n}
\end{equation}
(see \cite[pp.~25--26]{PS72}). For each $x\in\I$, we put $\lambda_E(x)\coloneqq\lambda((e_n(x))_{n\ge1})$ and $\lambda_P(x)\coloneqq\lambda((p_n(x))_{n\ge1})$. Since $p_n(x)\ge n$ for all $n\in\N$, we have $\lambda_P(x)\le1$ for all $x\in\I$.

\begin{proposition}\label{prop:exponent}
For every $x\in\I$, we have
\begin{equation}\label{eq:truncation}
 \lambda_P(\Theta(x))=\min\{\lambda_E(x),1\}.
\end{equation}
Moreover, the set $C$ in \eqref{eq:theta} satisfies $C\subseteq\{y\in\I:\lambda_P(y)=1\}$. Consequently,
\[
 \Theta\bigl(\{x\in\I:\lambda_E(x)=\alpha\}\bigr)=\{y\in\I:\lambda_P(y)=\alpha\}
\]
for each $\alpha\in[0,1)$, whereas
\[
 \Theta\bigl(\{x\in\I:\lambda_E(x)\ge1\}\bigr)=\{y\in\I:\lambda_P(y)=1\}\setminus C .
\]
\end{proposition}

\begin{proof}
Let $x\in\I$, and write $e_n\coloneqq e_n(x)$ and $p_n\coloneqq e_n+n-2=p_n(\Theta(x))$. For each integer $n\ge2$, put $M_n\coloneqq\max\{e_n,n\}$. Then $M_n\le p_n\le2M_n$, and hence
\[
 0\le
 \frac{\log n}{\log M_n}-\frac{\log n}{\log p_n}
 \le\frac{\log2}{\log n}.
\]
Since $\log n/\log M_n=\min\{\log n/\log e_n,1\}$, \eqref{eq:truncation} follows from \eqref{eq:exponent}. If $y\in C$, then there exists an integer $c$ such that $p_n(y)=n+c$ for all sufficiently large $n$, and so $\lambda_P(y)=1$. The two set identities follow from \eqref{eq:truncation}, since $\Theta$ is a bijection of $\I$ onto $\I\setminus C$ and $C\subseteq\{y\in\I:\lambda_P(y)=1\}$.
\end{proof}

By Proposition~\ref{prop:exponent}, for each $x\in\I$, the convergence exponent of the Pierce expansion digit sequence of $\Theta(x)$ is that of the Engel expansion digit sequence of $x$ truncated at $1$. This explains why, in the sense of Baire category, the typical value of the convergence exponent is $\infty$ for Engel expansions but $1$ for Pierce expansions. Indeed, since
\[
 \{x\in\I:\lambda_E(x)=\infty\}=\bigcap_{k\in\N}\left\{x\in\I:\sum_{n\ge1}e_n(x)^{-k}=\infty\right\},
\]
applying Theorem~\ref{thm:engel-series}\ref{thm:engel-series:ii} to $v_n(t)\coloneqq t^{-k}$ for each $k\in\N$ together with the Baire category theorem shows that this set is a dense $G_\delta$ subset of $\I$ (see also \cite[Lemma~3.8 and Theorem~3.5]{SW21}). Hence $\{y\in\I:\lambda_P(y)=1\}$ contains the dense set $\Theta(\{x\in\I:\lambda_E(x)=\infty\})$. Moreover, it is itself a $G_\delta$ subset of $\I$, since $\lambda_P\le1$ and \eqref{eq:exponent} give
\[
 \{y\in\I:\lambda_P(y)=1\}=\bigcap_{j\in\N}\bigcap_{N\ge2}\bigcup_{n\ge N}\left\{y\in\I:\frac{\log n}{\log p_n(y)}>1-\frac1j\right\}.
\]
Therefore, $\{y\in\I:\lambda_P(y)=1\}$ is a dense $G_\delta$ subset of $\I$, which recovers the case $\alpha=1$ of \cite[Theorem~1.8]{AhnIJNT}.

\end{document}